\documentclass[12pt]{amsart}
\usepackage{amssymb,amsmath,amsthm}
\usepackage[mathscr]{euscript}
\usepackage[margin=1in]{geometry}
\usepackage[backrefs]{amsrefs}
\usepackage[pagebackref,hypertexnames=false, colorlinks, citecolor=red, linkcolor=red]{hyperref}
\usepackage[backrefs]{amsrefs}

\newtheorem{theorem}{Theorem}[section]
\newtheorem{lemma}[theorem]{Lemma}

\theoremstyle{definition}
\newtheorem{definition}[theorem]{Definition}

\theoremstyle{remark}

\numberwithin{equation}{section}

\begin{document}

\title[Two-Weight Inequality for Essentially Well Localized Operators]{A Two-Weight Inequality for Essentially Well Localized Operators with General Measures}
\author[P. Benge]{Philip Benge}
\address{Philip Benge, School of Mathematics, Washington University in St. Louis, St. Louis, MO}
\email{benge@wustl.edu}

\begin{abstract}
We develop a new formulation of well localized operators as well as a new proof for the necessary and sufficient conditions to characterize their boundedness between $L^2(\mathbb{R}^n,u)$ and $L^2(\mathbb{R}^n,v)$ for general Radon measures $u$ and $v$.
\end{abstract}

\maketitle
\tableofcontents

\section{Introduction}
We consider the boundedness of the integral operator $$Tf(x)=\int_{\mathbb{R}^n}K(x,y)f(y)dy$$ acting from $L^2(\mathbb{R}^n,u)$ to $L^2(\mathbb{R}^n,v)$, that is, we want to characterize the following inequality $$\left\Vert Tf\right\Vert_{L^2(\mathbb{R}^n,v)}\lesssim \left\Vert f\right\Vert_{L^2(\mathbb{R}^n,u)}$$ for all $f\in L^2(\mathbb{R}^n,u)\equiv\{f:\int_{\mathbb{R}^n} \left\vert f\right\vert^2 u<\infty\}$. As is common in two-weight problems, we will consider the change of variables $d\sigma=\frac{1}{u}dx$, $F=\frac{f}{u}$ and $d\omega=vdx$, which allows us to instead characterize the boundedness of the operator $T(\sigma \cdot)$ from $L^2(\mathbb{R}^n,\sigma)$ to $L^2(\mathbb{R}^n,\omega)$, that is, we want to characterize the inequality $$\left\Vert T(\sigma f)\right\Vert_{L^2(\mathbb{R}^n,\omega)}\lesssim \left\Vert f\right\Vert_{L^2(\mathbb{R}^n,\sigma)}.$$ 

Nazarov, Treil and Volberg in \cite{NTV1} found necessary and sufficient conditions for this inequality in the case when $T$ is a so called well localized operator. The primary examples of such operators are band operators, the Haar shift, Haar multipliers and dyadic paraproducts as well as perfect dyadic operators. 

In this paper we develop a new characterization of well localized operators and provide a new proof showing necessary and sufficient conditions for their boundedness in terms of Sawyer type testing on the operator. As in \cite{NTV1}, we proceed with an axiomatic approach. Rather than assume that our operator $T$ can be represented as an integral operator (which is not always possible), we instead characterize the operator based on how it behaves on an orthonormal basis of Haar-type functions. However, our behavior of interest is a simple support condition.

For example, let $\mathcal{D}$ be the standard dyadic grid in $\mathbb{R}$ and for each dyadic interval $I\in\mathcal{D}$, let $I_R$ and $I_L$ denote the right and left halves of the interval, respectively. Define the Haar function $h_I^0\equiv\frac{1}{\sqrt{\left\vert I\right\vert}}\left({\bf{1}}_{I_R}-{\bf{1}}_{I_L}\right)$ and the averaging function $h_I^1\equiv\frac{1}{\left\vert I\right\vert}{\bf{1}}_I$. Then an operator $T$ is said to be lower triangularly localized if there exists a constant $r>0$ such that for all dyadic intervals $I,J\in\mathcal{D}$ with $\left\vert I\right\vert\leq2\left\vert J\right\vert$, we have $$\left\langle T({\bf{1}}_J),h_I^0\right\rangle=0$$ if $I\not\subset J^{(r)}$ or if $\left\vert I\right\vert\leq 2^{-r}\left\vert J\right\vert$ and $I\not\subset J$. We say that $T$ is well localized if both $T$ and $T^*$ are lower triangularly localized.

Given a sequence $b=\{b_I\}_{I\in\mathcal{D}}$ and a function $f\in L^2(\mathbb{R})$ we define the martingale transform $$T_b f\equiv\sum_{I\in\mathcal{D}}b_I\left\langle f,h_I^0\right\rangle h_I^0$$ and the paraproduct $$P_b f\equiv\sum_{I\in\mathcal{D}}b_I\left\langle f,h_I^1\right\rangle h_I^0.$$ A simple computation shows that these are both well localized with respect to the constant $r=1$. 

We also see that $$T_b h_I^0=b_I h_I^0\text{ and }T^*_b h_I^0=b_I h_I^0$$ as well as $$P_b h_I^0=\sum_{J\subsetneq I}b_J\left\langle h_I^0,h_J^1\right\rangle h_J^0\text{ and }P^*_b h_I^0=b_I h_I^1.$$ Thus these operators satisfy a nice support condition when applied to any Haar function, namely, if $T$ is any of the operators above, we have $\text{supp}(Th_I^0)\subset I$.

For an additional example, we let $S$ be the Haar shift operator defined by $$Sf\equiv\sum_{I\in\mathcal{D}}b_I\left\langle f,h_I^0\right\rangle \left(h_{I_R}^0-h_{I_L}^0\right).$$ Then $S$ is well localized with associated constant $r=2$, and we also have $\text{supp}(S h_I^0)\subset I$ and $\text{supp}(S^* h_I^0)\subset I^{(1)}$, where $I^{(1)}$ denotes the dyadic parent of $I$. In the following section, we formally define this support condition, and in section \ref{WL} we show that this condition is in fact the same as the well localized condition up to a change in the constant $r$.

\section{Definitions and Statement of Results}

To define our orthonormal basis, let $\mathcal{D}^n$ denote the dyadic grid in $\mathbb{R}^n$, and for any $F\in\mathcal{D}^n$, define $\mathcal{D}^{n}_{k}(F)\equiv\{F'\in\mathcal{D}^n:F'\subseteq F, \ell(F')=\frac{1}{2^k}\ell(F)\}$, where $\ell(F)$ denotes the side length of the cube $F$. We further define $\mathcal{D}^n(F)=\bigcup_{k=0}^{\infty}{D_{k}^n(F)}$. From \cite{Wilson}, we have the following lemma.
\begin{lemma}
Let $F\in\mathcal{D}^n$. Then there are $2^n-1$ pairs of sets $\{(E_{F,i}^1,E_{F,i}^2)\}_{i=1}^{2^n-1}$ such that
\begin{enumerate}
\item for each $i$, $|E_{F,i}^1|=|E_{F,i}^2|$;
\item for each $i$, $E_{F,i}^1$ and $E_{F,i}^2$ are non-empty unions of cubes from $\mathcal{D}^{n}(F)$;
\item for every $i\neq j$, exactly one of the following must hold:
\begin{enumerate}
\item $E_{F,i}^1\cup E_{F,i}^2$ is entirely contained in either $E_{F,j}^1$ or $E_{F,j}^2$;
\item $E_{F,j}^1\cup E_{F,j}^2$ is entirely contained in either $E_{F,i}^1$ or $E_{F,i}^2$;
\item $(E_{F,i}^1\cup E_{F,i}^2)\cap(E_{F,j}^1\cup E_{F,j}^2)=\emptyset$.
\end{enumerate}
\end{enumerate}
\end{lemma}

For simplicity, we let $E_{F,i}=E_{F,i}^1\cup E_{F,i}^2$ and we will define $$\mathcal{H}^n\equiv\{E_{F,i}:{F\in\mathcal{D}^n,1\leq i\leq 2^n-1}\}$$ to be the collection of all rectangles $E_{F,i}$. We note that for all $i$, $E_{F,i}\subseteq F$, however, $E_{F,i}\not\subseteq \bigcup_{k=1}^{\infty}D_{k}^{n}(F)$. We further note that for all $k=0,1,\ldots,n$, we have that $$F=\bigcup_{i=2^{k-1}}^{2^k-1}E_{F,i}.$$ For $r\geq 0$, we define $E_{F,i}^{(r)}$ to be the rectangle of volume $2^r\left\vert E_{F,i}\right\vert$ containing $E_{F,i}$. 

We now define the Haar function $h_{F,i}^{0}$ and the averaging function $h_{F,i}^1$ associated with $E_{F,i}$ by $$h_{F,i}^0=\frac{1}{\sqrt{|E_{F,i}|}}\bigg(\mathbf{1}_{E_{F,i}^2}-\mathbf{1}_{E_{F,i}^1}\bigg)$$ and $$h_{F,i}^1=\frac{1}{|E_{F,i}|}\mathbf{1}_{E_{F,i}}.$$ 
The functions $\{h_{F,i}^0\}_{F\in\mathcal{D}^n,1\leq i \leq 2^n-1}$ form an orthonormal basis for $L^2(\mathbb{R}^n)$.

Given a Radon measure $\sigma$, we let $$h_{F,i}^\sigma=\sqrt{\frac{\sigma(E_{F,i}^1)}{\sigma(E_{F,i})\sigma(E_{F,i}^2)}}{\bf{1}}_{E_{F,i}^2}-\sqrt{\frac{\sigma(E_{F,i}^2)}{\sigma(E_{F,i})\sigma(E_{F,i}^1)}}{\bf{1}}_{E_{F,i}^1}$$ be the weight adapted Haar function if $\sigma(E_{F,i}^1),\sigma(E_{F,i}^2)>0$ and we set $h_{F,i}^\sigma\equiv 0$ if either $\sigma(E_{F,i}^1)=0$ or $\sigma(E_{F,i}^2)=0$. 

We will impose the following structure on this operator: 
\begin{definition}
An operator is said to be {\it essentially well localized} if there exists an $r\geq 0$ such that for all $E_{F,i}$ the following properties hold:
\begin{equation}
\label{SimpleLemma}
\text{supp}(T(\sigma h_{F,i}^\sigma))\subseteq E_{F,i}^{(r)};\text{    }\text{supp}(T^*(\omega h_{F,i}^\omega))\subseteq E_{F,i}^{(r)}.
\end{equation}

\end{definition}

We will establish the two-weight boundedness for any Radon measures $\sigma$ and $\omega$ by adapting the proof strategy found in characterizing the two-weight inequality for the Hilbert Transform (see \cite{H} and \cite{L}). We now state our main theorem.

\begin{theorem}
Let $T$ be essentially well localized for some $r\geq0$. Let $\sigma$ and $\omega$ be two Radon measures on $\mathbb{R}^n$. Then 
\begin{equation}
\label{MainIneq}
\left\Vert T(\sigma f)\right\Vert_{L^2(\mathbb{R}^n,\omega)}\lesssim C \left\Vert f\right\Vert_{L^2(\mathbb{R}^n,\sigma)},\text{      }\left\Vert T^*(\omega f)\right\Vert_{L^2(\mathbb{R}^n,\sigma)}\lesssim C \left\Vert f\right\Vert_{L^2(\mathbb{R}^n,\omega)}
\end{equation}
if and only if for all $E_{F,i}\in\mathcal{H}^n$ and $E_{G,j}\in\mathcal{H}^n$ with $2^{-r}\left\vert E_{F,i}\right\vert\leq\left\vert E_{G,j}\right\vert\leq 2^{r}\left\vert E_{F,i}\right\vert$ and $E_{G,j}\cap E_{F,i}^{(r)}\neq\emptyset$, the following testing conditions hold:
\begin{align}
\label{TestingConditions}
\left\Vert{\bf{1}}_{E_{F,i}}T(\sigma {\bf{1}}_{E_{F,i}})\right\Vert_{L^2(\mathbb{R}^n,\omega)}&\lesssim C_1\left\Vert {\bf{1}}_{E_{F,i}}\right\Vert_{L^2(\mathbb{R}^n,\sigma)};\\
\left\Vert{\bf{1}}_{E_{F,i}}T^*(\omega {\bf{1}}_{E_{F,i}})\right\Vert_{L^2(\mathbb{R}^n,\sigma)}&\lesssim C_2\left\Vert {\bf{1}}_{E_{F,i}}\right\Vert_{L^2(\mathbb{R}^n,\omega)};\\
\label{WeakBoundedness}
\left\vert\left\langle T(\sigma {\bf{1}}_{E_{F,i}}),{\bf{1}}_{E_{G,j}}\right\rangle_\omega\right\vert&\lesssim C_3\sigma(E_{F,i})^{1/2}\omega(E_{G,j})^{1/2}.
\end{align} 
Moreover, we have that $C\simeq C_1+C_2+C_3$.
\end{theorem}

Now because each $E_{F,i}\in\mathcal{H}^n$ is a union of cubes $Q\in\mathcal{D}^n$ and the boundedness of $T$ would imply a similar testing condition on cubes, we immediately have the following result.

\begin{theorem}
\label{CubeLocal}
Let $T$ be essentially well localized for some $r\geq0$. Let $\sigma$ and $\omega$ be two Radon measures on $\mathbb{R}^n$. Then 
\begin{equation}
\label{CubeMainIneq}
\left\Vert T(\sigma f)\right\Vert_{L^2(\mathbb{R}^n,\omega)}\lesssim C \left\Vert f\right\Vert_{L^2(\mathbb{R}^n,\sigma)},\text{      }\left\Vert T^*(\omega f)\right\Vert_{L^2(\mathbb{R}^n,\sigma)}\lesssim C \left\Vert f\right\Vert_{L^2(\mathbb{R}^n,\omega)}
\end{equation}
if and only if for all $Q\in\mathcal{D}^n$ and $R\in\mathcal{D}^n$ with $2^{-r}\left\vert Q\right\vert\leq\left\vert R\right\vert\leq 2^{r}\left\vert Q\right\vert$ and $R\cap Q^{(r)}\neq\emptyset$, the following testing conditions hold:
\begin{equation}
\label{CubeTestingConditions}
\left\Vert{\bf{1}}_{Q}T(\sigma {\bf{1}}_{Q})\right\Vert_{L^2(\mathbb{R}^n,\omega)}\lesssim C_1\left\Vert {\bf{1}}_{Q}\right\Vert_{L^2(\mathbb{R}^n,\sigma)},\text{      }\left\Vert{\bf{1}}_{Q}T^*(\omega {\bf{1}}_{Q})\right\Vert_{L^2(\mathbb{R}^n,\sigma)}\lesssim C_2\left\Vert {\bf{1}}_{Q}\right\Vert_{L^2(\mathbb{R}^n,\omega)};
\end{equation}
\begin{equation}
\label{CubeWeakBoundedness}
\left\vert\left\langle T(\sigma {\bf{1}}_{Q}),{\bf{1}}_{R}\right\rangle_\omega\right\vert\lesssim C_3\sigma(Q)^{1/2}\omega(R)^{1/2}.
\end{equation} 
Moreover, we have that $C\simeq C_1+C_2+C_3$.
\end{theorem}

We are also able to easily extract global testing conditions by noting that the boundedness of $T$ immediately implies the global testing conditions \eqref{GlobalTestingConditions} below, which then imply the local testing conditions \eqref{TestingConditions} and \eqref{WeakBoundedness}. With this we state another corollary.

\begin{theorem}
Let $T$ be essentially well localized for some $r\geq0$. Let $\sigma$ and $\omega$ be two Radon measures on $\mathbb{R}^n$. Then 
\begin{equation}
\label{MainGlobalIneq}
\left\Vert T(\sigma f)\right\Vert_{L^2(\mathbb{R}^n,\omega)}\lesssim C \left\Vert f\right\Vert_{L^2(\mathbb{R}^n,\sigma)},\text{      }\left\Vert T^*(\omega f)\right\Vert_{L^2(\mathbb{R}^n,\sigma)}\lesssim C \left\Vert f\right\Vert_{L^2(\mathbb{R}^n,\omega)}
\end{equation}
if and only if for all $E_{F,i}\in\mathcal{H}^n$, the following testing conditions hold:
\begin{align}
\label{GlobalTestingConditions}
\left\Vert T(\sigma {\bf{1}}_{E_{F,i}})\right\Vert_{L^2(\mathbb{R}^n,\omega)}\lesssim C_1\left\Vert {\bf{1}}_{E_{F,i}}\right\Vert_{L^2(\mathbb{R}^n,\sigma)},\\
\left\Vert T^*(\omega {\bf{1}}_{E_{F,i}})\right\Vert_{L^2(\mathbb{R}^n,\sigma)}\lesssim C_2\left\Vert {\bf{1}}_{E_{F,i}}\right\Vert_{L^2(\mathbb{R}^n,\omega)}.
\end{align}
Moreover, we have that $C\simeq C_1+C_2$.
\end{theorem}

A similar result can be stated for cubes by \eqref{CubeLocal}.

\begin{theorem}
\label{CubeGlobal}
Let $T$ be essentially well localized for some $r\geq0$. Let $\sigma$ and $\omega$ be two Radon measures on $\mathbb{R}^n$. Then 
\begin{equation}
\label{CubeGlobalMainIneq}
\left\Vert T(\sigma f)\right\Vert_{L^2(\mathbb{R}^n,\omega)}\lesssim C \left\Vert f\right\Vert_{L^2(\mathbb{R}^n,\sigma)},\text{      }\left\Vert T^*(\omega f)\right\Vert_{L^2(\mathbb{R}^n,\sigma)}\lesssim C \left\Vert f\right\Vert_{L^2(\mathbb{R}^n,\omega)}
\end{equation}
if and only if for all $Q\in\mathcal{D}^n$, the following testing conditions hold:
\begin{equation}
\label{CubeGlobalTestingConditions}
\left\Vert T(\sigma {\bf{1}}_{Q})\right\Vert_{L^2(\mathbb{R}^n,\omega)}\lesssim C_1\left\Vert {\bf{1}}_{Q}\right\Vert_{L^2(\mathbb{R}^n,\sigma)},\text{      }\left\Vert T^*(\omega {\bf{1}}_{Q})\right\Vert_{L^2(\mathbb{R}^n,\sigma)}\lesssim C_2\left\Vert {\bf{1}}_{Q}\right\Vert_{L^2(\mathbb{R}^n,\omega)}.
\end{equation}
Moreover, we have that $C\simeq C_1+C_2$.
\end{theorem}

\section{Initial Reductions}

Whenever there is no ambiguity, we will simply write $L^2(\sigma)$ instead of $L^2(\mathbb{R}^n,\sigma)$. We will also write $\sum_{E_{F,i}}$ rather than $\sum_{F\in\mathcal{D}^n}\sum_{i:1\leq i\leq 2^n-1}$.
By duality, we will study the pairing $\left\langle Tf,g\right\rangle_\omega$, where $$\left\langle f,g\right\rangle_{\omega}=\int_{\mathbb{R}^n} fg\omega.$$ We will also consider the martingale expansions of $f$ and $g$ with respect to $\sigma$ and $\omega$, respectively. Namely, $f=\sum_{E_{F,i}}\Delta_{E_{F,i}}^{\sigma}f$ and $g=\sum_{E_{G,j}}\Delta_{E_{G,j}}^{\omega}g$ where $\Delta_{E_{F,i}}^\sigma f=\hat{f}_{\sigma}(E_{F,i})h_{{F,i}}^{\sigma}=\left\langle f,h_{{F,i}}^{\sigma}\right\rangle_\sigma h_{{F,i}}^\sigma$. 

We first make the assumption that $f$ and $g$ are finite linear combinations of indicator functions ${\bf{1}}_{E_{F,i}}$ where $2^{-d}\left\vert Q_0\right\vert\leq\left\vert E_{F,i}\right\vert\leq\left\vert Q_0\right\vert$ for some cube $Q_0$ and some $d>0$. We will obtain our estimates independent of $Q_0$ and $d$ and noting the density of simple functions in $L^2(\sigma)$ will give the result for general $f$ and $g$. 

We now want to reduce to considering functions $f$ and $g$ compactly supported on a dyadic cube $Q_0\in\mathcal{D}^n$. To do this, for $1\leq j\leq 2^n$, let $Q_j\in\mathcal{D}^n$ be dyadic cubes in the $j^\text{th}$ orthant, respectively, so that $Q_0\subseteq \bigcup_{j}Q_j$. Then we can write $f=\sum_{j}f{\bf{1}}_{Q_j}$ and similarly for $g$. So $\left\Vert f\right\Vert_{L^2(\sigma)}^2=\sum_{j}\left\Vert f{\bf{1}}_{Q_j}\right\Vert_{L^2(\sigma)}^2$. We now have
\begin{align*}
\left\langle T (\sigma f),g\right\rangle_\omega&=\sum_{i,j}\left\langle T (\sigma f{\bf{1}}_{Q_i}),g{\bf{1}}_{Q_j}\right\rangle_\omega.
\end{align*}
Analyzing the terms with $i\neq j$ gives $$\left\langle T (\sigma f{\bf{1}}_{Q_i}),g{\bf{1}}_{Q_j}\right\rangle_\omega=\sum_{E_{F,i}\cap Q_i\neq\emptyset}\sum_{E_{G,j}\cap Q_j\neq\emptyset}\left\langle T(\sigma \Delta_{E_{F,i}}^\sigma f),\Delta_{E_{G,j}}^\omega g\right\rangle_\omega$$
where the cubes $E_{F,i}$ are in the $i^\text{th}$ orthant and the cubes $E_{G,j}$ are in the $j^\text{th}$ orthant. Now by property \eqref{SimpleLemma}, this is zero by support considerations. Thus it suffices to show that we have $\left\vert\left\langle T (\sigma f{\bf{1}}_{Q_j}),g{\bf{1}}_{Q_j}\right\rangle_\omega\right\vert\lesssim C\left\Vert f{\bf{1}}_{Q_j}\right\Vert_{L^2(\sigma)}\left\Vert g{\bf{1}}_{Q_j}\right\Vert_{L^2(\omega)}$ because then we have 
\begin{align*}
\left\vert\left\langle T (\sigma f),g\right\rangle_\omega\right\vert=\left\vert\sum_j\left\langle T(\sigma f{\bf{1}}_{Q_j}),g{\bf{1}}_{Q,j}\right\rangle_\omega\right\vert&\lesssim C\sum_{j}\left\Vert f{\bf{1}}_{Q_j}\right\Vert_{L^2(\sigma)}\left\Vert g{\bf{1}}_{Q_j}\right\Vert_{L^2(\omega)}\\
&\lesssim C\left\Vert f\right\Vert_{L^2(\sigma)}\left\Vert g\right\Vert_{L^2(\omega)}.
\end{align*}
So with this, we assume $Q_0\in\mathcal{D}^n$. Now we can write $$f=\sum_{E_{F,i}\subset Q_0}\Delta_{E_{F,i}}^\sigma f + \left\langle f\right\rangle_{Q_0}^\sigma {\bf{1}}_{Q_0}$$ as well as $$\left\Vert f\right\Vert^2_{L^2(\sigma)}=\sum_{E_{F,i}\subset Q_0}\left\Vert\Delta_{E_{F,i}}^\sigma f\right\Vert^2_{L^2(\sigma)} + \left\vert\left\langle f\right\rangle_{Q_0}^\sigma\right\vert^2 \sigma(Q_0)$$ where $\left\langle f\right\rangle_{E_{F,i}}^\sigma=\frac{1}{\sigma(E_{F,i})}\int_{E_{F,i}}f\sigma$ is the average of $f$ with respect to $\sigma$. With this we have
\begin{align*}
\left\langle T (\sigma f),g\right\rangle_\omega&=\sum_{E_{F,i},E_{G,j}\subset Q_0}\left\langle T (\sigma \Delta_{E_{F,i}}^\sigma f),\Delta_{E_{G,j}}^\omega g\right\rangle_\omega+\sum_{E_{F,i}\subset Q_0}\left\langle T(\sigma \Delta_{E_{F,i}}^\sigma f), \left\langle g\right\rangle_{Q_0}^\omega {\bf{1}}_{Q_0}\right\rangle_\omega\\
&+\sum_{E_{G,j}\subset Q_0}\left\langle \left\langle f\right\rangle_{Q_0}^\sigma T(\sigma {\bf{1}}_{Q_0}),\Delta_{E_{G,j}}^\omega g\right\rangle_\omega+\left\langle\left\langle f\right\rangle_{Q_0}^\sigma T(\sigma{\bf{1}}_{Q_0}),\left\langle g\right\rangle_{Q_0}^\omega {\bf{1}}_{Q_0}\right\rangle_\omega.
\end{align*}

For the last three terms, we have the following lemma.
\begin{lemma}
The following estimates hold:
\begin{enumerate}
\item
$\left\vert\sum_{E_{F,i}\subset Q_0}\left\langle T(\sigma \Delta_{E_{F,i}}^\sigma f), \left\langle g\right\rangle_{Q_0}^\omega {\bf{1}}_{Q_0}\right\rangle_\omega\right\vert\lesssim C_2\left\Vert f\right\Vert_{L^2(\sigma)}\left\Vert g\right\Vert_{L^2(\omega)}$;\\
\item $\left\vert\sum_{E_{G,j}\subset Q_0}\left\langle \left\langle f\right\rangle_{Q_0}^\sigma T(\sigma {\bf{1}}_{Q_0}),\Delta_{E_{G,j}}^\omega g\right\rangle_\omega\right\vert\lesssim C_1\left\Vert f\right\Vert_{L^2(\sigma)}\left\Vert g\right\Vert_{L^2(\omega)}$;\\
\item $\left\vert\left\langle\left\langle f\right\rangle_{Q_0}^\sigma T(\sigma{\bf{1}}_{Q_0}),\left\langle g\right\rangle_{Q_0}^\omega {\bf{1}}_{Q_0}\right\rangle_\omega\right\vert\lesssim C_1\left\Vert f\right\Vert_{L^2(\sigma)}\left\Vert g\right\Vert_{L^2(\omega)}.$
\end{enumerate}
\end{lemma}
\begin{proof}
These are immediately controlled by Cauchy-Schwarz and applying the testing hypotheses. We will show the first estimate, and note that the remaining follow similarly.
\begin{align*}
\left\vert\sum_{E_{F,i}\subset Q_0}\left\langle T(\sigma \Delta_{E_{F,i}}^\sigma f), \left\langle g\right\rangle_{Q_0}^\omega {\bf{1}}_{Q_0}\right\rangle_\omega\right\vert&= \left\vert\left\langle g\right\rangle_{Q_0}^\omega\right\vert\left\vert\left\langle T\left(\sigma\sum_{E_{F,i}\subset Q_0}\Delta_{E_{F,i}}^\sigma f\right),{\bf{1}}_{Q_0}\right\rangle_\omega\right\vert\\
&\leq \left\vert\left\langle g\right\rangle_{Q_0}^\omega\right\vert\left\Vert\sum_{E_{F,i}\subset Q_0}\Delta_{E_{F,i}}^\sigma f\right\Vert_{L^2(\sigma)}\left\Vert {\bf{1}}_{Q_0} T^*(\omega{\bf{1}}_{Q_0})\right\Vert_{L^2(\sigma)}\\
&\lesssim C_2 \left\Vert f\right\Vert_{L^2(\sigma)}\left\Vert g\right\Vert_{L^2(\omega)}.
\end{align*}

\end{proof}
Thus it is enough to control only the first term, so we consider functions $f$ and $g$ that have mean zero with respect to $\sigma$ and $\omega$, respectively.

\section{Proof of Main Theorem}

Throughout the proof, we will use the notation $\Pi(f,g)=\left\langle T(\sigma f),g\right\rangle_{\omega}$. We have
\begin{align*}
\left\langle T(\sigma f),g\right\rangle_{\omega}&=\sum_{E_{F,i},E_{G,j}}\Pi(\Delta_{E_{F,i}}^\sigma f,\Delta_{E_{G,j}}^\omega g)\\
&=\left(\sum_{2^{-r}\left\vert E_{F,i}\right\vert\leq\left\vert E_{G,j}\right\vert \leq 2^r\left\vert E_{F,i}\right\vert}+\sum_{\left\vert E_{G,j}\right\vert>2^r\left\vert E_{F,i}\right\vert}+\sum_{\left\vert E_{F,i}\right\vert > 2^r\left\vert E_{G,j}\right\vert}\right)\Pi(\Delta_{E_{F,i}}^\sigma f,\Delta_{E_{G,j}}^\omega g)\\
&=\sum_{k}\sum_{E_{F,i}}\Pi(\Delta_{E_{F,i}}^\sigma f,\Delta_{E_{F,i,k}}^\omega g)+\sum_{E_{G,j}\supsetneq E_{F,i}^{(r)}}\Pi(\Delta_{E_{F,i}}^\sigma f,\Delta_{E_{G,j}}^\omega g)\\
&\indent+\sum_{E_{F,i}\supsetneq E_{G,j}^{(r)}}\Pi(\Delta_{E_{F,i}}^\sigma f,\Delta_{E_{G,j}}^\omega g)\\
&={\bf{A}}(f,g)+{\bf{B}}(f,g)+{\bf{C}}(f,g)
\end{align*}
where we have used property \eqref{SimpleLemma} in the third equality to only consider rectangles with containment and where $E_{F,i,k}$ is the $k^\text{th}$ rectangle $E_{G,j}$ such that $2^{-r}\left\vert E_{F,i}\right\vert\leq\left\vert E_{G,j}\right\vert \leq 2^r\left\vert E_{F,i}\right\vert$  and $\Pi(\Delta_{E_{F,i}}^\sigma f,\Delta_{E_{G,j}}^\omega g)\neq0$, for some ordering of this finite set. We note that all sets $E_{F,i,k}$ will be contained in $E_{F,i}^{(r)}$ and have length at least $2^{-r}\left\vert E_{F,i}\right\vert$, which gives that there are $M=M(r,n)=\left(2^{n(2r+1)}-1\right)/\left(2^n-1\right)$ such sets. We will consider the first two sums only, as the third sum is symmetric to ${\bf{B}}(f,g)$.
\begin{align*}
\left\vert{\bf{A}}(f,g)\right\vert&=\left\vert\sum_{k}\sum_{E_{F,i}\in\mathcal{D}^n}\Pi(\Delta_{E_{F,i}}^\sigma f,\Delta_{E_{F,i,k}}^\omega g)\right\vert\\
&=\left\vert\sum_{k}\sum_{{E_{F,i}}\in\mathcal{D}^n}\hat{f}_\sigma({E_{F,i}})\hat{g}_\omega({E_{F,i,k}})\Pi(h_{F,i}^\sigma,h_{F,i,k}^\omega)\right\vert\\
&\leq\sum_{k}\sup_{{E_{F,i}}}\left\vert \Pi(h_{F,i}^\sigma,h_{F,i,k}^\omega)\right\vert\left\Vert f\right\Vert_{L^2(\sigma)}\left\Vert g\right\Vert_{L^2(\omega)}\\
&\lesssim M C_3 \left\Vert f\right\Vert_{L^2(\sigma)}\left\Vert g\right\Vert_{L^2(\omega)}
\end{align*}
where we have used \eqref{WeakBoundedness} in the final inequality. We now need only estimate $${\bf{B}}(f,g)=\sum_{E_{G,j}\supsetneq E_{F,i}^{(r)}}\Pi(\Delta_{E_{F,i}}^\sigma f,\Delta_{E_{G,j}}^\omega g).$$
We now define suitable stopping rectangles in $\mathcal{H}^n$. We initialize our construction with $S_0\equiv Q_0$, and we let $\mathcal{S}\equiv\{S_0\}$. In the inductive step, for a minimal stopping rectangle $S$, we let $\text{ch}_{\mathcal{S}}(S)$ be the set of all maximal $\mathcal{H}^n$ children $S'$ of S such that the following holds: %at least one of the following hold:
\begin{equation}
\label{Stoppingg}
\frac{1}{\omega(S')}\int_{S'}\left\vert g\right\vert d\omega>2\frac{1}{\omega(S)}\int_{S}\left\vert g\right\vert d\omega.
\end{equation}
%or
%\begin{equation}
%\label{Stoppingop}
%\frac{1}{\omega(S')}\int_{S'}\left\vert T*({\bf{1}}_S d\omega\right\vert^2 d\sigma>4\frac{1}{\omega(S)}\int_{S}\left\vert \pi_b^*({\bf{1}}_S d\omega\right\vert^2 d\sigma.
%\end{equation}

We see immediately that 
\begin{align*}
\sum_{S'\in \text{ch}_{\mathcal{S}}(S)}\omega(S')\leq\frac{1}{2}\omega(S),
\end{align*}
which gives us the Carleson condition
\begin{align*}
\sum_{S\in {\mathcal{S}},S\subseteq Q}\omega(S)\leq 2\omega(Q).
\end{align*}
We note that by the well-known dyadic Carleson embedding theorem (see \cite{Chung}), this condition implies that for all $g\in L^2(\omega)$
\begin{equation}
\sum_{S\in\mathcal{S}}\omega(S)\left\vert\left\langle g\right\rangle_S^\omega\right\vert^2\lesssim\left\Vert g\right\Vert_{L^2(\omega)}^2.
\end{equation}

For every cube ${E_{F,i}}\subseteq Q_0$, we define the stopping parent $$\pi {E_{F,i}}\equiv \text{min}\{S\in\mathcal{S}:S\supseteq {E_{F,i}}\}.$$ Now define the projections $$P_S^{\omega}g=\sum_{{E_{G,j}}:\pi {E_{G,j}}=S}\Delta_{{E_{G,j}}}^\omega g,\text{    }\tilde{P}_S^{\sigma}f=\sum_{{E_{F,i}}:\pi {E_{F,i}^{(r)}}=S}\Delta_{{E_{F,i}}}^\sigma f.$$ So $f=\sum_{S\in\mathcal{S}}\tilde{P}_{S}^\sigma f$, and similarly $g=\sum_{S\in\mathcal{S}}{P}_{S}^\omega g$. With this, we have
\begin{align*}
{\bf{B}}(f,g)&=\sum_{S,S'\in\mathcal{S}}{\bf{B}}(\tilde{P}_S^\sigma f,P_{S'}^\omega g)\\
&=\sum_{S\in\mathcal{S}}{\bf{B}}(\tilde{P}_S^\sigma f,P_{S}^\omega g) +\sum_{S,S'\in\mathcal{S},S'\supsetneq S}{\bf{B}}(\tilde{P}_S^\sigma f,P_{S'}^\omega g)\\
&={\bf{B}}_1(f,g)+{\bf{B}}_2(f,g).
\end{align*}
We note that we do not get any contribution from the stoppping cubes $S\supsetneq S'$ because we are reduced to the case ${E_{F,i}}^{(r)}\subsetneq {E_{G,j}}$. We will now handle ${\bf{B}}_2(f,g)$ first:
\begin{align*}
{\bf{B}}_2(f,g)&=\sum_{S,S'\in\mathcal{S},S'\supsetneq S}{\bf{B}}(\tilde{P}_S^\sigma f,P_{S'}^\omega g)\\
&=\sum_{S\in\mathcal{S}}{\bf{B}}\left(\tilde{P}_S^\sigma f,\sum_{S'\in\mathcal{S},S'\supsetneq S}P_{S'}^\omega g\right)\\
&=\sum_{S\in\mathcal{S}}\sum_{{E_{G,j}}\supsetneq S}\Pi(\tilde{P}_S^\sigma f,{\bf{1}}_S\Delta_{E_{G,j}}^\omega g)\\
&=\sum_{S\in\mathcal{S}}\left\langle g\right\rangle_{S}^\omega \Pi(\tilde{P}_S^\sigma f,{\bf{1}}_{S}).
\end{align*}
In the third and fourth equalities, we have used $T(\sigma \tilde{P}_S^\sigma f)={\bf{1}}_S T(\sigma\tilde{P}_S^\sigma f)$ by property \eqref{SimpleLemma} and further that $$\left\langle g\right\rangle_S^\omega{\bf{1}}_S={\bf{1}}_S\sum_{E_{G,j}\supsetneq S}\Delta_{E_{G,j}}^\omega g.$$ With this we have
\begin{align*}
\left\vert {\bf{B}}_2(f,g)\right\vert&\leq \sum_{S\in\mathcal{S}}\left\vert\left\langle g\right\rangle_S^\omega \right\vert\left\Vert \tilde{P}_S^\sigma f\right\Vert_{L^2(\sigma)}\left\Vert{\bf{1}}_S T^*(\omega{\bf{1}}_S)\right\Vert_{L^2(\sigma)}\\
&\leq C_2 \left\Vert f\right\Vert_{L^2(\sigma)}\left(\sum_{S\in\mathcal{S}}\left\vert\left\langle g\right\rangle_S^\omega\right\vert^2\omega(S)\right)^{1/2}\\
&\lesssim C_2 \left\Vert f\right\Vert_{L^2(\sigma)}\left\Vert g\right\Vert_{L^2(\omega)}
\end{align*}
where the last inequality follows by the Carleson Embedding Theorem.

This leaves us only needing to estimate the term ${\bf{B}}_1(f,g)$. First, we will set $B_{S}(f,g)={\bf{B}}(\tilde{P}_S^\sigma f,P_S^\omega g)$. Then we have that ${\bf{B}}_{1}(f,g)=\sum_{S\in\mathcal{S}}B_S(f,g)$. We now have
\begin{align*}
B_S(f,g)&= \sum_{{E_{F,i}^{(r)}}\subsetneq {E_{G,j}}\subseteq S, \pi {E_{F,i}^{(r)}}=\pi {E_{G,j}}=S} \Pi(\Delta_{{E_{F,i}}}^\sigma f,\Delta_{E_{G,j}}^\omega)\\
&=\sum_{{E_{F,i}^{(r)}}\subseteq S, \pi {E_{F,i}^{(r)}}=S}\Pi\left(\Delta_{E_{F,i}}^\sigma f,{\bf{1}}_{E_{F,i}^{(r)}}\sum_{{E_{G,j}}:{E_{F,i}^{(r)}}\subsetneq {E_{G,j}}\subseteq S,\pi {E_{G,j}}=S}\Delta_{E_{G,j}}^\omega g\right)\\
&=\sum_{{E_{F,i}^{(r)}}\subseteq S, \pi {E_{F,i}^{(r)}}=S}\Pi\left(\Delta_{E_{F,i}}^\sigma f,\left\langle g\right\rangle_{E_{F,i}^{(r)}}^\omega{\bf{1}}_{E_{F,i}^{(r)}}-\left\langle g\right\rangle_S^\omega{\bf{1}}_S\right)\\
&=\sum_{{E_{F,i}^{(r)}}\subseteq S, \pi {E_{F,i}^{(r)}}=S}\Pi\left(\Delta_{E_{F,i}}^\sigma f,\left\langle g\right\rangle_{E_{F,i}^{(r)}}^\omega{\bf{1}}_{E_{F,i}^{(r)}}\right)-\sum_{{E_{F,i}^{(r)}}\subseteq S, \pi {E_{F,i}^{(r)}}=S}\Pi\left(\Delta_{E_{F,i}}^\sigma f,\left\langle g\right\rangle_S^\omega{\bf{1}}_S\right)\\
&=\bf{I-II}.
\end{align*}
Recalling by the stopping condition \eqref{Stoppingg} we have that $\left\vert\left\langle g\right\rangle_{E_{F,i}^{(r)}}\right\vert\leq\left\langle\left\vert g\right\vert\right\rangle_{E_{F,i}^{(r)}}\leq 2\left\langle\left\vert g\right\vert\right\rangle_S.$ With this, we have for the first term
\begin{align*}
\left\vert \bf{I}\right\vert&=\left\vert \sum_{{E_{F,i}^{(r)}}\subseteq S, \pi {E_{F,i}^{(r)}}=S}\Pi\left(\Delta_{E_{F,i}}^\sigma f,\left\langle g\right\rangle_{E_{F,i}^{(r)}}^\omega{\bf{1}}_{E_{F,i}^{(r)}}\right)\right\vert\\
&\lesssim \sum_{{E_{F,i}^{(r)}}\subseteq S, \pi {E_{F,i}^{(r)}}=S}\left\langle \left\vert g\right\vert\right\rangle_S^\omega\left\vert\Pi\left(\Delta_{E_{F,i}}^\sigma f,{\bf{1}}_{E_{F,i}^{(r)}}\right)\right\vert\\
&\lesssim \left\langle \left\vert g\right\vert\right\rangle_S^\omega\sum_{{E_{F,i}^{(r)}}\subseteq S, \pi {E_{F,i}^{(r)}}=S}\left\Vert\Delta_{E_{F,i}}^\sigma f\right\Vert_{L^2(\sigma)}\left\Vert{\bf{1}}_{E_{F,i}^{(r)}} T^*(\omega {\bf{1}}_{E_{F,i}^{(r)}})\right\Vert_{L^2(\sigma)}\\
&\lesssim C_2\left\langle \left\vert g\right\vert\right\rangle_S^\omega\sum_{{E_{F,i}^{(r)}}\subseteq S, \pi {E_{F,i}^{(r)}}=S}\left\Vert\Delta_{E_{F,i}}^\sigma f\right\Vert_{L^2(\sigma)}\left\Vert{\bf{1}}_{E_{F,i}^{(r)}}\right\Vert_{L^2(\omega)}\\
&\lesssim C_2\left\langle \left\vert g\right\vert\right\rangle_S^\omega \left(\sum_{{E_{F,i}^{(r)}}\subseteq S, \pi {E_{F,i}^{(r)}}=S}\left\Vert\Delta_{E_{F,i}}^\sigma f\right\Vert^2_{L^2(\sigma)}\right)^{1/2}\left(\sum_{{E_{F,i}^{(r)}}\subseteq S, \pi {E_{F,i}^{(r)}}=S}\omega({E_{F,i}})\right)^{1/2}\\
&\lesssim C_2\left\langle \left\vert g\right\vert\right\rangle_S^\omega \omega(S)^{1/2}\left\Vert \tilde{P}_S f\right\Vert_{L^2(\sigma)}.
\end{align*}
For the second term, we have
\begin{align*}
\left\vert \bf{II}\right\vert&\leq\left\langle\left\vert g\right\vert\right\rangle_S^\omega\left\vert\sum_{{E_{F,i}^{(r)}}\subseteq S, \pi {E_{F,i}^{(r)}}=S}\Pi(\Delta_{E_{F,i}}^\sigma f,{\bf{1}}_{S})\right\vert\\
&=\left\langle\left\vert g\right\vert\right\rangle_S^\omega\left\vert\Pi\left(\sum_{{E_{F,i}^{(r)}}\subseteq S, \pi {E_{F,i}^{(r)}}=S}\Delta_{E_{F,i}}^\sigma f,{\bf{1}}_{S}\right)\right\vert\\
&\leq \left\langle\left\vert g\right\vert\right\rangle_S^\omega\left\Vert {\bf{1}}_S T^*(\omega {\bf{1}}_S) \right\Vert_{L^2(\sigma)}\left\Vert \tilde{P}_S^\sigma f\right\Vert_{L^2(\sigma)}\\
&\lesssim C_2\left\langle\left\vert g\right\vert\right\rangle_S^\omega\omega(S)^{1/2}\left\Vert \tilde{P}_S^\sigma f\right\Vert_{L^2(\sigma)}.\\
\end{align*}
With this, we have finally that
\begin{align*}
\left\vert{\bf{B}_1}(f,g)\right\vert&=\left\vert\sum_{S\in\mathcal{S}}B_S(f,g)\right\vert\\
&\lesssim C_2 \sum_{S\in\mathcal{S}}\left\langle\left\vert g\right\vert\right\rangle_S^\omega\omega(S)^{1/2}\left\Vert \tilde{P}_S^\sigma f\right\Vert_{L^2(\sigma)}\\
&\lesssim C_2\left\Vert f\right\Vert_{L^2(\sigma)}\left(\sum_{S\in\mathcal{S}}\left(\left\langle\left\vert g\right\vert\right\rangle_S^\omega\right)^2\omega(S)\right)^{1/2}\\
&\lesssim C_2\left\Vert f\right\Vert_{L^2(\sigma)}\left\Vert g\right\Vert_{L^2(\omega)}
\end{align*}
where we again use the Carleson Embedding Theorem in the last inequality.

So we have now established $C\lesssim C_1+C_2+C_3$. To obtain the other inequality, we notice that we have $$C\gtrsim\left\Vert T\right\Vert_{L^2(\sigma)\rightarrow L^2(\omega)}\gtrsim\sup_{{E_{F,i}}}\frac{\left\Vert{\bf{1}}_{E_{F,i}} T(\sigma{\bf{1}}_{{E_{F,i}}})\right\Vert_{L^2(\omega)}}{\left\Vert{\bf{1}}_{{E_{F,i}}}\right\Vert_{L^2(\sigma)}}=C_1.$$
Similarly, we have $C\gtrsim\left\Vert T\right\Vert_{L^2(\sigma)\rightarrow L^2(\omega)}\gtrsim C_2$ and $C\gtrsim\left\Vert T\right\Vert_{L^2(\sigma)\rightarrow L^2(\omega)}\gtrsim C_3$. This indeed gives $$C\simeq C_1+C_2+C_3.$$

\section{Well Localized Operators}
\label{WL}

We recall from \cite{NTV1} that well localized operators have the following definition:

\begin{definition}
$T$ is said to be lower triangularly localized if there exists a constant $r>0$ such that for all cubes $R$ and $Q$ with $\left\vert R\right\vert\leq2\left\vert Q\right\vert$ and for all $\omega$-Haar functions on $R$ $h_R^\omega$, we have $$\left\langle T(\sigma{\bf{1}}_Q),h_R^\omega\right\rangle_\omega=0$$ if $R\not\subset Q^{(r)}$ or if $\left\vert R\right\vert\leq 2^{-r}\left\vert Q\right\vert$ and $R\not\subset Q$.

We say that $T$ is well localized if both $T$ and $T^*$ are lower triangularly localized.
\end{definition}

We will now show that well localized operators are essentially well localized. Let $T$ be a well localized operator associated with some $r>0$. Fix a cube $E_{F,i}$ and let $E_{G,j}$ be any cube with $\left\vert E_{G,j}\right\vert=\left\vert E_{F,i}\right\vert$ and $E_{F,i}^{(r)}\cap E_{G,j}=\emptyset$. So $$T (\sigma h_{F,i}^\sigma){\bf{1}}_{E_{G,j}}=\sum_{E_{H,k}\subset E_{G,j}}\Delta_{H,k}^\omega T (\sigma h_{F,i}^\sigma) +\left\langle T (\sigma h_{F,i}^\sigma)\right\rangle_{E_{G,j}}^\omega{\bf{1}}_{E_{G,j}}.$$ Now since $T$ is well localized, it is immediate that $\left\langle T (\sigma h_{F,i}^\sigma)\right\rangle_{E_{G,j}}^\omega=0$. So we have $$T (\sigma h_{F,i}^\sigma){\bf{1}}_{E_{G,j}}=\sum_{E_{H,k}\subset E_{G,j}}\Delta_{H,k}^\omega T (\sigma h_{F,i}^\sigma).$$ Now $$\Delta_{H,k}^\omega T (\sigma h_{F,i}^\sigma)=\frac{\sqrt{\sigma(E_{F,i}^1)}}{\sqrt{\sigma(E_{F,i})\sigma(E_{F,i}^2)}}\Delta_{H,k}^\omega T(\sigma {\bf{1}}_{E_{F,i}^2})-\frac{\sqrt{\sigma(E_{F,i}^2)}}{\sqrt{\sigma(E_{F,i})\sigma(E_{F,i}^1)}}\Delta_{H,k}^\omega T(\sigma {\bf{1}}_{E_{F,i}^1}).$$ For $E_{H,k}\subseteq E_{G,j}$, we clearly have $\left\vert E_{H,k}\right\vert\leq2\left\vert E_{F,i}^1\right\vert$ and $\left\vert E_{H,k}\right\vert\leq2\left\vert E_{F,i}^2\right\vert$. So applying the well localized property to each term gives $\Delta_{H,k}^\omega T (\sigma h_{F,i}^\sigma)=0$. 

Now for any rectangle $Q$ with $Q \cap E_{F,i}^{(r)}=\emptyset$, we can write $Q\subseteq\bigcup Q_k$ where $Q_k\cap E_{F,i}^{(r)}=\emptyset$ and $\left\vert Q_k\right\vert =\left\vert E_{F,i}\right\vert$. So we have $T(\sigma h_{F,i}^\sigma){\bf{1}}_Q={\bf{1}}_{Q}\sum_{k}T(\sigma h_{F,i}^\sigma){\bf{1}}_{Q_k}=0$. A similar calculation shows $T^*(\omega h_{F,i}^\omega){\bf{1}}_Q=0$. So we have that $T$ is essentially well localized.

This computation also gives the following characterization for essentially well localized operators.

\begin{theorem}
An operator $T$ is essentially well localized for some $r\geq 0$ if and only if for all $Q\in\mathcal{D}^n$ and $E_{F,i}\in\mathcal{H}^n$ with $\left\vert E_{F,i}\right\vert\leq2\left\vert Q\right\vert$ and $E_{F,i}\not\subset Q^{(r+1)}$, we have $$\left\langle T(\sigma{\bf{1}}_Q),h_{F,i}^\omega\right\rangle_\omega=0$$ and $$\left\langle T^*(\omega{\bf{1}}_Q),h_{F,i}^\sigma\right\rangle_\sigma=0.$$
\end{theorem}

However, if $\left\vert E_{F,i}\right\vert\leq 2^{-(r+1)}\left\vert Q\right\vert$ and $E_{F,i}\not\subset Q$, then we have that $E_{F,i}^{(r)}\cap Q=\emptyset$. With this, we immediately have the following characterization of essentially well localized operators.

\begin{theorem}
An operator $T$ is essentially well localized for some $r\geq 0$ if and only if $T$ is well localized for $r+1$.
\end{theorem}

Having an alternate characterization for well localized operators allows us to easily classify some operators as the following example shows. 

\begin{definition}
An operator $T$ is said to be an essentially perfect dyadic operator if for some $r\geq 0,$ $$T(\sigma f)(x)=\int_{\mathbb{R}} K(x,y)f(y)\sigma(y)dy$$ for $x\not\in\text{supp}(f)$, where $$K(x,y)\leq\frac{1}{\left\vert x-y\right\vert}$$ and $$\left\vert K(x,y)-K(x,y')\right\vert +\left\vert K(x,y)-K(x',y)\right\vert=0$$ whenever $x,x'\in I\in\mathcal{D}$, $y,y'\in J\in\mathcal{D}$ where $I^{(r)}\cap J=\emptyset$ and $I\cap J^{(r)}=\emptyset$.
\end{definition}
If $r=0$, we recover the perfect dyadic operators first introduced in \cite{AHMTT}.

Let $T:L^2(\sigma)\rightarrow L^2(\omega)$ be an essentially perfect dyadic operator and let $f=\sigma h_I^\sigma$ and $x\not\in I^{(r)}$. Then
for all $y,y'\in I$, we have $\left\vert K(x,y)-K(x,y')\right\vert=0$. So $K(x,\cdot)$ is constant on $I$. With this we have
\begin{align*}
T(\sigma h_I^\sigma)(x)&=\int_{I} K(x,y)h_I^\sigma(y)\sigma(y)dy\\
&=0.
\end{align*}

We can write $$T^*(\omega g)(y)=\int_{\mathbb{R}} \overline{K(x,y)}g(x)\omega(x)dx.$$ Now let $g=\omega h_J^\omega$ and let $y\not\in J^{(r)}$. Then for all $x,x'\in J$, we have $\left\vert \overline{K(x,y)}-\overline{K(x',y)}\right\vert=\left\vert K(x,y)-K(x',y)\right\vert=0$. So $\overline{K(\cdot,y)}$ is constant on $J$. As above, we have
\begin{align*}
T^*(\omega h_J^\omega)(y)&=\int_{J} \overline{K(x,y)}h_J^\omega(x)\omega(x)dx\\
&=0.
\end{align*}
So we have that $T$ is essentially well localized.

\section*{Acknowledgment}

The author would like to thank Brett Wick for his helpful conversations and support (NSF DMS grants \#1560955  and \#1603246).

\end{document}